\documentclass[twoside,leqno,12pt]{amsart} 
\usepackage{amsfonts} 
\usepackage{amsmath} 
\usepackage{amscd} 
\usepackage{amssymb} 
\usepackage{amsthm} 
\usepackage{url}
\usepackage{latexsym} 
\usepackage{color} 
\newtheorem{theorem}{Theorem} 
\newtheorem{lemma}{Lemma} 
\newtheorem{Proposition}{Proposition}

\numberwithin{equation}{section} 

\usepackage{bbm}
\usepackage{mathrsfs}

\usepackage{xcolor}

\begin{document} 
\title{A Bombieri-Vinogradov theorem for sectors in real quadratic number fields}

\author[S. Baier] {Stephan Baier}
\address{Stephan Baier\\ 
	Ramakrishna Mission Vivekananda Educational and Research Institute\\
	Department of Mathematics\\
	G.\ T.\ Road, PO~Belur Math, Howrah, West Bengal~711202\\
	India}
\email{stephanbaier2017@gmail.com}

\author[E. A. Molla] {Esrafil Ali Molla}
\address{Esrafil Ali Molla\\ Statistics and Mathematics Unit\\
Indian Statistical Institute\\
203 B.T. Road, Kolkata~700108\\
India}
\email{esrafil.math@gmail.com}

\date{\today}

\subjclass[2020]{Primary: 11R11, 11R44; Secondary: 11N05, 11N13}

\keywords{Bombieri-Vinogradov theorem, prime ideals, sectors, arithmetic progressions, Hecke characters, large sieve}

\begin{abstract}
We establish a Bombieri-Vinogradov theorem for sectors in real quadratic number fields.
\end{abstract}

\maketitle

\tableofcontents

\section{Introduction and main result}
\subsection{History}
The classical Bombieri-Vinogradov theorem is a powerful result about averages of the error term in the prime number theorem for arithmetic progressions, asserting that
$$
\sum\limits_{q\le Q}\max\limits_{(a,q)=1} \max\limits_{\substack{y\le x}}\left|\pi(y;q,a)-\frac{1}{\varphi(q)}\cdot \int\limits_2^y\frac{dt}{\log t}\right|\ll_A \frac{x}{\log^A x} + Qx^{1/2}\log^6 Qx
$$
for any $A>0$, where $\pi(y;q,a)$ counts the number of primes $p\le y$ in the residue class $a \bmod q$. In particular, if $Q\le x^{1/2-\varepsilon}$, then we get a bound of $O(x\log^{-A} x)$. The exponent $1/2-\varepsilon$ is commonly referred to as "level of distribution". This theorem has been extended to number fields $\mathbb{K}$ in different ways. Here, care needs to be taken because there may be {\it infinitely many} prime elements $p$ with bounded norm over $\mathbb{Q}$ in the ring of integers $\mathcal{O}_{\mathbb{K}}$ satisfying a congruence relation of the form $p\equiv a\bmod{\mathfrak{q}}$ with $\mathfrak{q}$ being an ideal in $\mathcal{O}_{\mathbb{K}}$. 

One way to get around this is to replace congruences $a\equiv b\bmod{m}$ of integers by ray class equivalences $\mathfrak{a} \sim \mathfrak{b}\bmod{\mathfrak{m}}$ of ideals to a modulus $\mathfrak{m}$ for $\mathbb{K}$, as introduced by Hecke. A generalized Bombieri-Vinogradov type theorem of this kind was established by Wilson \cite{Wil1969} and improved by Huxley \cite{Hux1971}.
One may also proceed in a different direction if $\mathbb{K}$ is Galois over $\mathbb{Q}$: Instead of prime ideals in ray classes, one may count rational primes $p\equiv a\bmod{q}$ which are unramified in $\mathbb{K}$ and whose Artin symbol $(p, \mathbb{K}/\mathbb{Q})$ equals a given conjugacy class $C$ in the Galois group of $\mathbb{K}$ over $\mathbb{Q}$.  This route was taken by Murty and Murty \cite{MurtyMurty1985} and later generalized to the situation of non-Galois extensions $\mathbb{K}$ of a general base field by Murty and Peterson \cite{MurtyPeter2013}. The congruence condition above can be expressed in the form $\mathcal{N}(\mathfrak{p})\equiv a\bmod{q}$, where $\mathcal{N}(\mathfrak{p})$ is the norm of a prime ideal $\mathfrak{p}$ in $\mathcal{O}_\mathbb{K}$.  The above results in \cite{MurtyMurty1985} and \cite{MurtyPeter2013} were strengthened by Jiang, L\"u and Wang \cite{JiangLuWang2021} and refined by Thorner \cite{Thorner2016} who included a short interval condition. Khale, O'Kuhn, Panidapu, Sun and Zhang \cite{Khale et al.2021} considered versions for Galois extensions of $\mathbb{Q}$ with a small sector condition determined by Hecke characters.  However, also in this setting, the congruence condition was of the form $\mathcal{N}(\mathfrak{p})\equiv a \bmod{q}$. 

Hinz \cite{Hinz1988} developed a version of the Bombieri-Vinogradov theorem in which he considered a genuine congruence condition of the form $p\equiv a\bmod{\mathfrak{q}}$ in $\mathcal{O}_\mathbb{K}$. So he really counted prime elements $p\in \mathcal{O}_{\mathbb{K}}$ in residue classes $a \bmod{\mathfrak{q}}$ such that the norm of $p$ over $\mathbb{Q}$ is bounded by $x$. For this to work, he  restricted these prime elements $p$ to a box. More in detail, his restriction was that $\left(p_1,p_2,...,p_{r_1},|p_{r_1+1}|^2,|p_{r_1+2}|^2,...,|p_{r_1+r_2}|^2\right)$ lies in a box $(0,y_1]\times \cdots \times (0,y_{r_1+r_2}]$, where $p_1,...,p_{r_1}$ are the real and $p_{r_1+1},\overline{p_{r_1+1}},...,p_{r_1+r_2},\overline{p_{r_1+r_2}}$ the complex embeddings of $p$. He assumed that the $y_k$'s are about of size $x^{1/(r_1+r_2)}$ so that the volume of the box is about of size $x$. 

Coleman and Swallow \cite{ColemanSwallow2005} derived a version of the Bombieri-Vinogradov theorem for {\it imaginary quadratic fields} with a genuine congruence condition $p\equiv a \bmod{\mathfrak{q}}$ and a restriction of $p$ to a small {\it sector}. More precisely, they restricted both the argument of $p$ and its norm over $\mathbb{Q}$ to a short interval. In this article, we establish a Bombieri-Vinogradov theorem with similar congruence and sector conditions for {\it real quadratic fields}. This case is more delicate since the unit group is infinite. Moreover, in this context, sectors are generally harder to handle than boxes, as considered by Hinz: summation formulae for algebraic integers are considerably more effective over boxes than over sectors because boxes are product sets whereas sectors are not. Consequently, in our setting of sectors we pass from algebraic integers to integral ideals and  detect both the congruence and the sector condition using Hecke characters, which are defined on the group of fractional ideals. The main difficulty here is that the infinitude of the unit group may cause the group of Hecke characters with trivial infinite part modulo $\mathfrak{q}$ to be much smaller than the group of Dirichlet characters modulo $\mathfrak{q}$, reflecting the fact that many residue classes may collapse into a single ray class under the action of units. We therefore carefully arrange a suitable larger set of Hecke characters modulo $\mathfrak{q}$, possibly with non-trivial infinite part, to detect residue classes. To formulate a meaningful result, we restrict the prime elements $p$ to a sector contained in a fundamental domain of $\mathbb{R}^2$ under multiplication by units in $\mathcal{O}_{\mathbb{K}}$. This may be viewed as more natural than a restriction to a box, as considered by Hinz. 

In this article, we will fix the sector and consider a full interval $[1,x]$ bounding the norm of $(p)$. Extending our result to small sectors and short intervals would require  zero density estimates for Hecke $L$-functions. These are avoided here to keep the arguments as simple as possible, only using a Siegel-Walfisz-type theorem and a large sieve inequality for a family of Hecke characters. An extension to small sectors and short intervals may be considered in future research.  

\subsection{General notation}
To state our main result, we first set up some notation.
For a number field  $\mathbb{K}$, we denote by $h$ its class number, by $\mathcal{O}_\mathbb{K}$ its ring of algebraic integers, by $\mathcal{O}_\mathbb{K}^{\ast}$ the group of units in $\mathcal{O}_\mathbb{K}$, by $\mathcal{I}_{\mathbb{K}}$ the set of non-zero integral ideals in $\mathcal{O}_{\mathbb{K}}$, by $\mathcal{P}_{\mathbb{K}}$ the subset of principal ideals, by $\mathcal{I}^{\text{frac}}_{\mathbb{K}}$ the group of fractional ideals and by $\mathcal{P}^{\text{frac}}_{\mathbb{K}}$ the subgroup of principal fractional ideals. For $\mathfrak{a},\mathfrak{b}\in \mathcal{I}_\mathbb{K}$, we write $(\mathfrak{a},\mathfrak{b})=1$ if $\mathfrak{a}$ and $\mathfrak{b}$ have no common prime ideal divisors. For $a\in \mathcal{O}_{\mathbb{K}}$ and $\mathfrak{a}\in \mathcal{I}_{\mathbb{K}}$, we write $(a,\mathfrak{a})=1$ iff $((a),\mathfrak{a})=1$. If $\mathfrak{a}\in \mathcal{I}_{\mathbb{K}}$, then we define the norm of the ideal $\mathfrak{a}$ by $\mathcal{N}(\mathfrak{a}):=[\mathcal{O}_\mathbb{K}:\mathfrak{a}]$.
If $a\in \mathbb{K}$, then we set $\mathcal{N}(a):=|N_{\mathbb{K}:\mathbb{Q}}(a)|$. We note that $\mathcal{N}(a)=\mathcal{N}((a))$ if $a\in \mathcal{O}_\mathbb{K}$.    
We define the von Mangoldt function on $\mathcal{I}_{\mathbb{K}}$ by 
$$ 
\Lambda(\mathfrak{a}):=\begin{cases} \log \mathcal{N}(\mathfrak{p}) & \mbox{ if } \mathfrak{a}=\mathfrak{p}^k \mbox{ for some prime ideal } \mathfrak{p} \mbox{ and } k\in \mathbb{N},\\ 0 & \mbox{ otherwise} 
\end{cases} 
$$ 
and set
$$
\Lambda (a):=\Lambda ((a))
$$
if $a\in \mathcal{O}_\mathbb{K}$. We define the Euler totient function on $\mathcal{I}_\mathbb{K}$ by
$$
\varphi\left(\mathfrak{a}\right):=\sharp\left(\mathcal{O}_\mathbb{K}/\mathfrak{a}\right)^{\ast}
$$
and set
$$
\varphi\left(a\right):=\varphi((a))
$$
if $a\in \mathcal{O}_\mathbb{K}$. If $\mathcal{C}$ is a condition, we define its indicator function as 
$$
I_{\mathcal{C}}:= \begin{cases} 1 & \mbox{if } \mathcal{C} \mbox{ is satisfied, } \\
0 & \mbox{otherwise.}
\end{cases}
$$
If $\mathcal{M}$ is a set, we define its indicator function as
$$
I_{\mathcal{M}}(x):= \begin{cases} 1 & \mbox{if } x\in \mathcal{M},  \\
0 & \mbox{otherwise.}
\end{cases}
$$

From now onward, we restrict $\mathbb{K}$ to be a real quadratic number field, i.e. $\mathbb{K}=\mathbb{Q}(\sqrt{d})$, where $d>1$ is a square-free integer. 
%Throughout the sequel, for simplicity, we shall also assume that $\mathbb{K}$ has class number one, i.e. $\mathcal{O}_{\mathbb{K}}$ is a principal ideal domain. 
We denote the two embeddings of $\mathbb{K}$, the identity and conjugation, by 
$$ 
\sigma_1(s+t\sqrt{d}):=s+t\sqrt{d} 
$$ 
and 
$$ 
\sigma_2(s+t\sqrt{d}):=s-t\sqrt{d}, 
$$ 
where $s,t\in \mathbb{Q}$. 
We recall that, by the Dirichlet unit theorem for the case of real quadratic fields $\mathbb{K}$, the unit group $\mathcal{O}_\mathbb{K}^{\ast}$ is generated by a fundamental unit $\epsilon>1$ and $-1$. Note that $\sigma_2(\epsilon)=\pm \epsilon^{-1}$ because
$$
|\epsilon \sigma_2(\epsilon)|=|\sigma_1(\epsilon)\sigma_2(\epsilon)|=\mathcal{N}(\epsilon)=1.
$$
We will use the following lemma to suitably restrict generators of principal ideals.

\begin{lemma} \label{sigma12}
Let $\mathfrak{a}\in \mathcal{P}_\mathbb{K}$. Then there exists a generator $a$ of $\mathfrak{a}$, unique up to the sign, such that
$$
\epsilon^{-1}|\sigma_2(a)|<|\sigma_1(a)|\le \epsilon|\sigma_2(a)|,
$$
where $\epsilon$ is the fundamental unit in $\mathcal{O}_\mathbb{K}$. In this case, we have
$$
\epsilon^{-1/2}\mathcal{N}(\mathfrak{a})^{1/2}\le |\sigma_i(a)|\le \epsilon^{1/2} \mathcal{N}(\mathfrak{a})^{1/2} \quad \mbox{for } i=1,2. 
$$ 
\end{lemma}
\begin{proof}
See \cite[Lemma 1]{BDM24}.
\end{proof}
We will formulate our Bombieri-Vinogradov theorem for sectors of the form 
\begin{equation} \label{def S}
S=S(\eta_1,\eta_2):=\left\{a\in \mathcal{O}_\mathbb{K} : \epsilon^{\eta_1}|\sigma_2(a)|< |\sigma_1(a)|\le \epsilon^{\eta_2} |\sigma_2(a)| \mbox{ and } a>0 \right\},
\end{equation}
where $\eta_1$ and $\eta_2$ are fixed such that $-1\le \eta_1< \eta_2\le 1$. By Lemma \ref{sigma12}, every $\mathfrak{a}\in \mathcal{P}_{\mathbb{K}}$ has exactly one generator in $S(-1,1)$. Throughout the following, let $x\ge 2$. For $a\in \mathcal{O}_{\mathbb{K}}$ and $\mathfrak{q}\in \mathcal{I}_{\mathbb{K}}$ with $(a,\mathfrak{q})=1$, we set
\begin{equation} \label{def psi2}
\psi_S(x;\mathfrak{q},a):=\sum\limits_{\substack{s\in S(\eta_1,\eta_2;x)\\  s\equiv a \bmod \mathfrak{q}}} \Lambda (s),
\end{equation}
where 
\begin{equation} \label{Seta1eta2x}
S(\eta_1,\eta_2;x):= \{s\in S(\eta_1,\eta_2): \mathcal{N}(s)\le x \}.
\end{equation}
Heuristically, we expect that $\psi_S(x;\mathfrak{q},a)$ is close to 
$$
M_S(x;\mathfrak{q},a):=\frac{\eta}{\varphi(\mathfrak{q})} \sum\limits_{\substack{s\in S(-1,1;x)\\ (s,\mathfrak{q})=1}} \Lambda(s),
$$
where we set 
$$
\eta:=\eta_2-\eta_1.
$$
We denote the error term of this approximation by 
\begin{equation}\label{def ES}
E_S(x;\mathfrak{q},a):= \psi_S(x;\mathfrak{q},a)-M_S(x;\mathfrak{q},a).
\end{equation}

\subsection{Main result}
Our main result is the following bound for the average of $E_S(x;\mathfrak{q},a)$ over residue classes. 
\begin{theorem}[Bombieri-Vinogradov theorem for sectors in real quadratic number fields] \label{main} Let $\mathbb{K}$ be a real quadratic field. 
Let $x\ge 2$, $1\le Q\le x^{4/5}$, $\varepsilon>0$ and $A>0$. Keep the notations above. Then
\begin{equation}\label{th main eq1}
\sum\limits_{\substack{\mathfrak{q}\in\mathcal{I}_{\mathbb{K}} \\ \mathcal{N}(\mathfrak{q})\le Q}}\,\, \max\limits_{\substack{a\bmod \mathfrak{q} \\(a,\mathfrak{q})=1}} \,\, \max\limits_{y\le x}\left| E_S(y;\mathfrak{q},a) \right| \ll_{\varepsilon,A,\eta,\mathbb{K}}  \left(x^{9/11}Q^{8/11}+ x^{1/2}Q^{2}\right) x^{\varepsilon}+\frac{x}{\log^A x}.
\end{equation}
In particular, if $Q\le x^{1/4-2\varepsilon}$, we have
\begin{equation*}
\sum\limits_{\substack{\mathfrak{q}\in\mathcal{I}_{\mathbb{K}} \\ \mathcal{N}(\mathfrak{q})\le Q}}\,\, \max\limits_{\substack{a\bmod \mathfrak{q} \\(a,\mathfrak{q})=1}} \,\, \max\limits_{y\le x}\left| E_S(y;\mathfrak{q},a) \right| \ll_{\varepsilon,A,\eta,\mathbb{K}} \frac{x}{\log^A x}.
\end{equation*}
\end{theorem}
{\bf Remark 1.} Using Lemma \ref{sigma12} and the Chebotarev density theorem, we may further approximate the main term $M_S(x;\mathfrak{q},a)$ by 
\begin{equation*}
\begin{split}
M_S(x;\mathfrak{q},a)=\frac{\eta}{\varphi(\mathfrak{q})} \sum\limits_{\substack{\mathfrak{s}\in \mathcal{P}_\mathbb{K}\\\ \mathcal{N}(\mathfrak{s})\le x\\ (\mathfrak{s},\mathfrak{q})=1}} \Lambda(\mathfrak{s})=& \frac{\eta}{\varphi(\mathfrak{q})} \Bigg(\sum\limits_{\substack{\mathfrak{s}\in \mathcal{P}_\mathbb{K}\\\ \mathcal{N}(\mathfrak{s})\le x}} \Lambda(\mathfrak{s}) + O\left(\log^2(\mathcal{N}(\mathfrak{q}x))\right)\Bigg)\\ = &
\frac{\eta}{h\varphi(\mathfrak{q})} \left( x + O\left( \frac{x}{(\log x)^A} \right) \right)
\end{split}
\end{equation*}
for any $A>0$, provided that $\mathcal{N}(\mathfrak{q})\le x$. Therefore, the estimates in Theorem \ref{main} remain valid if the term $E_S(y;\mathfrak{q},a)$ is replaced by 
$$
E_S'(y;\mathfrak{q},a):=\psi_S(y;\mathfrak{q},a)-\frac{\eta}{h\varphi(\mathfrak{q})} \cdot x.
$$

\subsection{Organization of this article} After introducing the necessary tools, we will use Fourier analysis to pick out the residue class and sector conditions via Hecke characters. This reduces the problem to bounding averages of Hecke character sums over prime ideals (in fact, we work with the von Mangoldt function instead of the indicator function of primes). The moduli $\mathfrak{q}$ with a small norm will be treated using a version of the Siegel-Walfisz theorem. For the remaining contribution, we will use Vaughan's identity to reduce the problem to bounding averages of bilinear sums with Hecke characters over ideals. For the estimation of these averages, we will use a version of the Polya-Vinogradov inequality as well as a large sieve inequality for Hecke characters. This will finally allow us to establish our Bombieri-Vinogradov theorem for sectors in real quadratic fields. Throughout the remainder of this article, we allow implied constants to depend on $\varepsilon$, $A$, $\eta$ and $\mathbb{K}$. \\ \\
{\bf Remark 2:} The level of distribution of $1/4-\varepsilon$ in our Theorem \ref{main} arises from a large sieve inequality for a family of Hecke characters (Lemma \ref{LSieve} below) which is worse than what we expect to hold (see Remark 3 after the proof of Lemma \ref{LSieve}). It would be desirable to improve this bound. Our conjectural large sieve inequality would lead to a level of distribution $1/2-\varepsilon$.    
\section{Preliminaries}

\subsection{Hecke characters for real quadratic fields}
In this subsection, we provide an explicit description of Hecke characters for real quadratic fields, based on our more detailed description in \cite[section 2]{BDM24}, which in turn was based on material in \cite{Hecke1918}, \cite{Hecke1920}, \cite[section 3.3]{Miya}, and \cite{mult}. Hecke characters are characters of the group of fractional ideals of $\mathbb{K}$ which are coprime to some modulus. We will use them to detect elements of $\mathcal{O}_{\mathbb{K}}$ in intersections of sectors and arithmetic progressions. To describe them properly, we set up some notation below.

Let $\mathfrak{q}\in \mathcal{I}_{\mathbb{K}}$. If $a,b\in \mathbb{K}$, we say that $a$ and $b$ are multiplicatively congruent modulo $\mathfrak{q}$, denoted by $a\equiv^{\ast} b\bmod{\mathfrak{q}}$, if for every prime ideal $\mathfrak{p}$ dividing $\mathfrak{q}$, the $\mathfrak{p}$-adic valuation satisfies $v_{\mathfrak{p}}((a-b))\ge v_{\mathfrak{p}}(\mathfrak{q})$. We say that a fractional ideal $\mathfrak{a}\subset \mathbb{K}$ is coprime to $\mathfrak{q}$ if $v_{\mathfrak{p}}(\mathfrak{a})=0$ for all prime ideal divisors $\mathfrak{p}$ of $\mathfrak{q}$. 
Let $\tilde{\chi}$ be a Dirichlet character modulo $\mathfrak{q}$ (the  pull-back of a group character for $(\mathcal{O}_\mathbb{K}/\mathfrak{q})^{\ast}$ to 
$\mathcal{O}_\mathbb{K}$). Its multiplicative extension to $\mathbb{K}$ is defined as 
$$
\chi(\alpha):=\tilde{\chi}(a) \quad \mbox{if } a\in \mathcal{O}_\mathbb{K} \mbox{ and } \alpha \equiv^{\ast} a \bmod{\mathfrak{q}} 
$$ 
when $\alpha\in \mathbb{K}$ is such that $(\alpha)$ is coprime to $\mathfrak{q}$. If $(\alpha)$ is not coprime to $\mathfrak{q}$, then we set $\chi(\alpha):=0$. It is easy to see that $\chi(\alpha)$ is well-defined. 
We refer to $\chi$ as a Dirichlet character modulo $\mathfrak{q}$ as well. On the group $\mathcal{P}^{\text{frac}}_{\mathbb{K}}$ of principal fractional ideals of $\mathbb{K}$, Hecke characters modulo $\mathfrak{q}$ take the form
\begin{equation} \label{Heckedirichlet}
\chi_{\text{Hecke}}((\alpha))=\chi(\alpha)\chi_{\infty}(\alpha),
\end{equation}
where $\chi$ is a Dirichlet character modulo $\mathfrak{q}$ and $\chi_{\infty}$ is referred to as infinite part of $\chi_{\text{Hecke}}$. Every Hecke character $\chi_{\text{Hecke}}$ on $\mathcal{P}^{\text{frac}}_{\mathbb{K}}$ extends to $h$ distinct Hecke characters on the full group of fractional ideals $\mathcal{I}^{\text{frac}}_{\mathbb{K}}$. By convention, given a Hecke character $\chi_{\text{Hecke}}$ on $\mathcal{P}^{\text{frac}}_{\mathbb{K}}$, we will fix one of these extended characters and denote it by $\chi_{\text{Hecke}}$ as well. We say that the Hecke character $\chi_{\text{Hecke}}$ {\it belongs} to the Dirichlet character $\chi$. (More precisely, $\chi_{\text{Hecke}}$ should be viewed as a Hecke character to a modulus with finite part $\mathfrak{q}$ and infinite part $\mathfrak{m}_{\infty}$. For simplicity, we surpress $\mathfrak{m}_{\infty}$ throughout this article.) For $\chi_{\text{Hecke}}((\alpha))$ to be well-defined, it is necessary that
\begin{equation} \label{requirement}
\chi(u)\chi_\infty(u) = 1 \quad \text{for all units } u \in \mathcal{O}_\mathbb{K}^{\ast}.
\end{equation}
Specifically, for a real quadratic number field $\mathbb{K}$, the infinite part takes the form (see \cite[equation (2.3)]{BDM24})
\begin{equation}
\chi_\infty(\alpha)=\left|\frac{\sigma_1(\alpha)}{\sigma_2(\alpha)}\right|^{iv} \mbox{sgn}\left(\sigma_1(\alpha)\right)^{u_1} \mbox{sgn}\left(\sigma_2(\alpha)\right)^{u_2},
\end{equation}
where $v\in \mathbb{R}$ and $u_1,u_2\in \{0,1\}$. The requirement in \eqref{requirement} imposes certain restrictions on these parameters $v,u_1,u_2$ which we work out below.

The unit group $\mathcal{O}_\mathbb{K}^{\ast}$ is generated by the fundamental unit $\epsilon>1$ and $-1$. Therefore it suffices to ensure that $\chi\chi_{\infty}(\epsilon)=1$ and $\chi\chi_{\infty}(-1)=1$.
Suppose that 
\begin{equation} \label{wdef}
\chi(\epsilon)\mbox{sgn}\left(\sigma_2(\epsilon)\right)^{u_2}= \exp(-2\pi i\rho). 
\end{equation}
Then $\rho$ takes the form
\begin{equation} \label{rhodef}
\rho=\frac{g}{2}-\frac{\arg(\chi(\epsilon))}{2\pi},
\end{equation}
where 
\begin{equation} \label{gdef}
g:=\begin{cases} 0 & \mbox{ if } \mbox{sgn}\left(\sigma_2(\epsilon)\right)^{u_2}=1,\\
1 & \mbox{ if } \mbox{sgn}\left(\sigma_2(\epsilon)\right)^{u_2}=-1. \end{cases}
\end{equation}
Since $\sigma_1(\epsilon)=\epsilon>0$, we have
$\mbox{sgn}\left(\sigma_1(\epsilon)\right)^{u_1}=1$. Thus, to ensure that $\chi\chi_{\infty}(\epsilon)=1$, it suffices that 
$$\mbox{sgn}(\sigma_2(\epsilon))^{-u_2}\chi_\infty(\epsilon)=\left|\frac{\sigma_1(\epsilon)}{\sigma_2(\epsilon)}\right|^{iv}= \exp(2\pi i\rho),
$$
which is the case iff
\begin{equation*}
\epsilon^{2iv}=\exp(2iv\log \epsilon)=\exp(2\pi i\rho). 
\end{equation*}
This in turn is equivalent to the condition
$$
v = \frac{\pi(n + \rho)}{\log \epsilon} \quad \text{for some } n \in \mathbb{Z}.
$$
Moreover, to ensure that $\chi\chi_\infty(-1) = 1$, it suffices that $u_1,u_2\in \{0,1\}$ satisfy 
\begin{equation} \label{u12cond}
\begin{cases}
u_1 = u_2 & \mbox{if } \chi(-1) = 1,\\
u_1 \neq u_2 & \mbox{if } \chi(-1) = -1.
\end{cases}
\end{equation}
Combining the above, the infinite part $\chi_\infty$ takes the form
\begin{equation} \label{infinitypart}
\chi_\infty(\alpha)=\left|\frac{\sigma_1(\alpha)}{\sigma_2(\alpha)}\right|^{i\pi(n+\rho)/\log \epsilon}\cdot\mbox{sgn}(\sigma_1(\alpha))^{u_1}\cdot\mbox{sgn}(\sigma_2(\alpha))^{u_2}
\end{equation}
with $\rho$ as defined in \eqref{rhodef} and $u_1, u_2 \in \{0,1\}$ satisfying the condition \eqref{u12cond}.

A particular Hecke character is the principal Hecke character $\chi_{\text{Hecke},0}$ modulo $\mathfrak{q}$ which equals the principal Dirichlet character $\chi_0$ modulo $\mathfrak{q}$. So in this case, the infinite part of $\chi_{\text{Hecke},0}$ is trivial. 
  
\subsection{Sums with Hecke Characters}
We shall use the following bound for sums of Hecke characters due to Landau which generalizes the classical Polya-Vinogradov estimate to number fields.

\begin{Proposition}[Landau's generalization of the Polya-Vinogradov inequality] \label{Landau} Let $X\ge 1$ and $\mathbb{K}$ be a number field of degree $n$ over $\mathbb{Q}$. Let $\chi_{\text{\rm Hecke}}$ be a non-trivial Hecke character modulo an ideal $\mathfrak{q}\in \mathcal{P}_{\mathbb{K}}$. Then 
$$
\sum\limits_{\substack{\mathfrak{a} \in \mathcal{I}_{\mathbb{K}}\\ \mathcal{N}(\mathfrak{a})\le X}} \chi_{\rm Hecke}(\mathfrak{a}) \ll \mathcal{N}(\mathfrak{q})^{1/(n+1)} X^{(n-1)/(n+1)}\log^n 2\mathcal{N}(\mathfrak{q}).
$$
\end{Proposition}

\begin{proof} See \cite{Lan}.
\end{proof}
The case relevant for us will be when $n=2$. We shall also use the following generalization of the Siegel-Walfisz theorem for number fields.
\begin{Proposition}[Siegel-Walfisz theorem for number fields] \label{Siegel}
Let $X\ge 2$, $A>0$ and $\mathbb{K}$ be a number field. Then there exists a positive constant $C=C(A,\mathbb{K})$ such that for any primitive Hecke character $\chi_{\text{\rm Hecke}}$ modulo an ideal $\mathfrak{q}\in \mathcal{P}_{\mathbb{K}}$ with $\mathcal{N}(\mathfrak{q})\le \log^A X$, we have
\begin{equation} \label{in}
\sum\limits_{\substack{\mathfrak{a} \in \mathcal{I}_\mathbb{K} \\ \mathcal{N}(\mathfrak{a}) \le X}}\chi_{\text{\rm Hecke}}(\mathfrak{a}) \Lambda(\mathfrak{a})\ll \frac{X}{\exp(C\sqrt{\log X})}.
\end{equation}
\end{Proposition}
\begin{proof}
See \cite[Chapter 5]{IwKo}.
\end{proof}

\subsection{Vaughan's identity} 
To treat sums of the form in \eqref{in} if $\mathcal{N}(\mathfrak{q})>\log^A x$, we use Vaughan’s identity for $\mathcal{I}_{\mathbb{K}}$ to decompose the von Mangoldt function. 
\begin{lemma}[Vaughan's identity for number fields] \label{Vaughanideals} 
Let $\mathfrak{n}\in \mathcal{I}_\mathbb{K}$ and $U, V\ge 1$ such that $U < \mathcal{N}(\mathfrak{n})$. Then 
$$ 
\Lambda(\mathfrak{n})= a_1(\mathfrak{n}) + a_2(\mathfrak{n})+ a_3(\mathfrak{n}), 
$$ 
where  
$$ 
a_1(\mathfrak{n}):=- \sum\limits_{\substack{ \mathfrak{r}\mathfrak{m}\mathfrak{l}=\mathfrak{n}\\ \mathcal{N}(\mathfrak{m}) \le V\\ \mathcal{N}(\mathfrak{l}) \le U}} \mu(\mathfrak{m}) \Lambda(\mathfrak{l}), 
$$ 
$$ 
a_2(\mathfrak{n}):=\sum\limits_{\substack{\mathfrak{l} \mathfrak{r}=\mathfrak{n} \\ \mathcal{N}(\mathfrak{l}) \le V}}\mu(\mathfrak{l})\log( \mathcal{N}(\mathfrak{r})) 
$$ 
and 
$$ a_3(\mathfrak{n}):= -\sum\limits_{\substack{ \mathfrak{m}\mathfrak{r}=\mathfrak{n} \\ \mathcal{N}(\mathfrak{m}) > U \\ \mathcal{N}(\mathfrak{r}) >1}} \Lambda(\mathfrak{m}) \sum\limits_{\substack{\mathfrak{d} |\mathfrak{r} \\ \mathcal{N}(\mathfrak{d}) \le V}} \mu(\mathfrak{d}), 
$$ 
with the sums above running over ideals in $\mathcal{I}_{\mathbb{K}}$.  
\end{lemma} 
\begin{proof}
The proof is analogous to that of the classical Vaughan identity. For the details, see \cite[Page 194]{Bru}.
\end{proof}

Let $\mathbf{G} :\mathcal{I}_{\mathbb{K}} \to \mathbb{C}$ be a function such that $\mathbf{G}(\mathfrak{a})=0$ if $\mathcal{N}(\mathfrak{a})\le U$. Then it follows from Lemma \ref{Vaughanideals} that 
\begin{equation} \label{Lambdasplit} 
\sum\limits_{\mathcal{N}(\mathfrak{n})\le X} \Lambda(\mathfrak{n}) \mathbf{G}(\mathfrak{n}) =S_1 + S_2 + S_3, 
\end{equation} 
where $$ S_i= \sum\limits_{\mathcal{N}(\mathfrak{n})\le X} \mathbf{G}(\mathfrak{n}) a_i(\mathfrak{n}).$$ 
We may bound the sums $S_1$ and  $S_2$ by (compare with \cite[page 95]{Bru})  
\begin{equation} \label{theS1}
\begin{split} 
S_1 &\ll (\log UV) \sum\limits_{ \mathcal{N}(\mathfrak{t})\le UV} \Bigg| \sum\limits_{\substack{ \mathcal{N}(\mathfrak{r}) \le X/\mathcal{N}(\mathfrak{t})}} \mathbf{G}(\mathfrak{t}\mathfrak{r}) \Bigg|
\end{split} 
\end{equation} 
and 
\begin{equation} \label{theS2}
\begin{split} 
S_2 &\ll (\log X) \sum\limits_{\mathcal{N}(\mathfrak{l}) \le V} \max\limits_{w\le X/\mathcal{N}(l)} \Bigg| \sum\limits_{\substack{ \mathcal{N}(\mathfrak{r})\le w}} \mathbf{G}(\mathfrak{l}\mathfrak{r})\Bigg|,
\end{split} 
\end{equation}
where we remove the term $\log \mathcal{N}(\mathfrak{r})$ implicit in the sum $S_2$ using partial summation.  
The sum $S_3$ can be expressed in the form
\begin{equation} \label{theS3} 
\begin{split} 
S_3 & = - \sum\limits_{U<\mathcal{N}(\mathfrak{m})\le X/ V} \quad \sum\limits_{V < \mathcal{N}(\mathfrak{r}) \le X/ \mathcal{N}(\mathfrak{m})} \Lambda(\mathfrak{m})H(\mathfrak{r}) \mathbf{G}(\mathfrak{m}\mathfrak{r}), 
\end{split} 
\end{equation} 
where we set
\begin{equation} \label{Hdef} 
H(\mathfrak{r}):=\sum\limits_{\substack{\mathfrak{d}|\mathfrak{r}\\\mathcal{N}(\mathfrak{d})\le V}}\mu(\mathfrak{d}) 
\end{equation} 
and note that $H(\mathfrak{r}) = 0$ for $1<\mathcal{N}(\mathfrak{r})\le V$. 

\subsection{Duality principle} 
The following lemma will be used to derive a large sieve inequality for Hecke characters.
\begin{lemma}[Duality principle] \label{Duality}
Let $\Phi : \mathbb{Z}^2 \to \mathbb{C}$ be a complex-valued function. Suppose that for any sequence of complex numbers $(\beta_n)_{n \le N}$,
$$
\sum\limits_{m\le M}\Big|\sum\limits_{n\le N}\beta_n \Phi(m,n)\Big|^2\le \Delta(M,N) \sum\limits_{n\le N} |\beta_n|^2.
$$
Then for any sequence of complex numbers $\alpha_m$,
$$
\sum\limits_{n\le N}\Big|\sum\limits_{m\le M}\alpha_m \Phi(m,n)\Big|^2\le \Delta(M,N) \sum\limits_{m\le M} |\alpha_m|^2,
$$
where $\Delta(M,N)$ is the same quantity in both inequalities.
\end{lemma}
\begin{proof}
See \cite[section $7.1$]{IwKo}.
\end{proof}

\subsection{Cosmetic surgery}
The following lemma, termed "cosmetic surgery" by some authors, will be used to remove certain summation conditions in the form of inequalities between variables. 
\begin{lemma}[Cosmetic surgery] \label{LemmaCosmetic}
For any $T\ge 1$ and two distinct real numbers $\alpha, \beta>0$, we have
$$
1_{\alpha<\beta}= \int_{-T}^{T}e^{it\alpha}\frac{\sin t\beta}{\pi t}dt +O\Big(\frac{1}{T|\beta-\alpha|}\Big),
$$
where the implied constant is absolute.
\end{lemma}
\begin{proof}
See \cite[Lemma 2.2]{HarmanBook}.
\end{proof}

\section{Reduction to sums with Hecke characters}

\subsection{Picking out residue classes} \label{Picking}
Using the orthogonality relations for Dirichlet characters, we begin with writing $\psi_S(y;\mathfrak{q},a)$ in the form
\begin{equation}
\psi_S(y;\mathfrak{q},a)=\frac{1}{\varphi(\mathfrak{q})}\sum\limits_{\chi\bmod\mathfrak{q}}\overline{\chi(a)} \psi_S(y,\chi),
\end{equation}
where 
\begin{equation}\label{Def psiS 0}
\psi_S(y,\chi):=\sum\limits_{s\in S(\eta_1,\eta_2;y)} \chi(s)\Lambda(s).
\end{equation}
Next, we set
\begin{equation}\label{Def E tilde}
\Tilde{E}_S(y;\mathfrak{q},a):= \psi_S(y;\mathfrak{q},a)-\frac{1}{\varphi(\mathfrak{q})}\sum\limits_{\substack{s\in S(\eta_1,\eta_2;y)\\ (s,\mathfrak{q})=1}} \Lambda(s)= \frac{1}{\varphi(\mathfrak{q})}\sum\limits_{\substack{\chi\bmod\mathfrak{q}\\ \chi\not=\chi_0}}\overline{\chi(a)} \psi_S(y,\chi),  
\end{equation}
where $\chi_0$ is the principal Dirichlet character modulo $\mathfrak{q}$. Our strategy for the proof of Theorem \ref{main} will be to work out average bounds for $\Tilde{E}_S(y;\mathfrak{q},a)$ and finally approximate $\Tilde{E}_S(y;\mathfrak{q},a)$ by $E_S(y;\mathfrak{q},a)$, defined in \eqref{def ES}.  To this end, we set
\begin{equation}
\tilde{E}_S(y;\mathfrak{q}):= \max\limits_{\substack{a\bmod \mathfrak{q}\\(a,\mathfrak{q})=1}}|\Tilde{E}_S(y;\mathfrak{q},a)|
\end{equation}
and
\begin{equation}
E_S^\ast(x;\mathfrak{q}):=\max_{y\le x}\tilde{E}_S(y;\mathfrak{q}).
\end{equation}  
Our immediate aim is to bound the sum
$$
\sum\limits_{\substack{\mathfrak{q}\in \mathcal{I}_{\mathbb{K}}\\ 2\le \mathcal{N}(\mathfrak{q})\le Q}} 
E_S^\ast(x;\mathfrak{q}).
$$
Here, the summation condition $\mathcal{N}(\mathfrak{q})\ge 2$ comes from the fact that $E_S^{\ast}(x;(1))=0$ since there is no non-principal Dirichlet character modulo $(1)$.

By the triangle inequality, we have
\begin{equation} \label{Bdd of E tilde}
|\Tilde{E}_S(y;\mathfrak{q},a)|\le \frac{1}{\varphi(\mathfrak{q})}\sum\limits_{\substack{\chi \bmod \mathfrak{q}\\ \chi \neq \chi_0}} |\psi_S(y,\chi)|.  
\end{equation}
Since the right-hand side is independent of the residue class $a$ modulo $\mathfrak{q}$,  it follows that
\begin{equation}\label{E* bdd1}
E_S^\ast(x;\mathfrak{q})\le\frac{1}{\varphi(\mathfrak{q})}\sum\limits_{\substack{\chi \bmod \mathfrak{q}\\ \chi\neq\chi_0}} \max_{y\le x}|\psi_S(y,\chi)|. 
\end{equation}

\subsection{Reduction to primitive characters}
Let $\chi \bmod \mathfrak{q}$ be a non-principal Dirichlet character and $\chi_1 \bmod \mathfrak{q}_1$ with $\mathfrak{q}_1|\mathfrak{q}$ be the primitive character inducing $\chi$. We observe that $\mathfrak{q}_1\not=(1)$
since $\chi$ is non-principal. It is easy to see that 
$$
\psi_S(y,\chi)=\psi_S(y,\chi_1)+O\left(\log^2(\mathcal{N}(\mathfrak{q})y)\right). 
$$
Consequently, from \eqref{E* bdd1}, we get 
\begin{equation*}
\sum\limits_{2\le \mathcal{N}(\mathfrak{q})\le Q}E^\ast(x;\mathfrak{q})\ll\sum\limits_{2\le \mathcal{N}(\mathfrak{q}_1)\le Q} \;\; \sideset{}{^\ast}\sum\limits_{\chi_1 \bmod \mathfrak{q}_1}\max\limits_{y\le x}|\psi_S(y,\chi_1)|\Bigg(\sum\limits_{1\le\mathcal{N}(\mathfrak{j})\le Q/\mathcal{N}(\mathfrak{q}_1)}\frac{1}{\varphi(\mathfrak{j}\mathfrak{q}_1)}\Bigg)+Q\log^2(Qx),
\end{equation*}
where the asterisk indicates that the relevant summation is restricted to primitive characters. 
We note that $\varphi(\mathfrak{j}\mathfrak{q}_1)\ge\varphi(\mathfrak{j})\varphi(\mathfrak{q})$, which implies 
$$
\sum\limits_{1\le\mathcal{N}(\mathfrak{j})\le Q/\mathcal{N}(\mathfrak{q}_1)} \frac{1}{\varphi(\mathfrak{j}\mathfrak{q})}\le\frac{1}{\varphi(\mathfrak{q})}\sum\limits_{1\le\mathcal{N}(\mathfrak{j})\le Q/\mathcal{N}(\mathfrak{q}_1)} \frac{1}{\varphi(\mathfrak{j})} \ll \frac{\log Q}{\varphi(\mathfrak{q})}.
$$
Assuming $Q\le x$, it follows that
\begin{equation}\label{E*primitive}
\sum\limits_{2\le \mathcal{N}(\mathfrak{q})\le Q} E^\ast(x;\mathfrak{q})\ll (\log x)\sum\limits_{2\le \mathcal{N}(\mathfrak{q})\le Q}\frac{1}{\varphi(\mathfrak{q})} \;\; \sideset{}{^\ast}\sum\limits_{\chi \bmod \mathfrak{q}}\max\limits_{y\le x}|\psi_S(y,\chi)|+Q\log^2 x.
\end{equation}

\subsection{Writing Dirichlet characters in terms of Hecke characters}
For a primitive Dirichlet character $\chi$ modulo $\mathfrak{q}$, let $\chi_{\text{Hecke}}$ be a Hecke character belonging to $\chi$.  Recalling \eqref{Heckedirichlet} and \eqref{Def psiS 0}, we may express $\psi_S(y,\chi)$ in the form
\begin{equation} \label{defpsichi2}
\psi_S(y,\chi):=\sum\limits_{s\in S(\eta_1,\eta_2;y)}\Lambda(s) \chi_{\text{Hecke}}((s))\overline{\chi}_\infty(s).
\end{equation}
Recalling \eqref{infinitypart} and $s>0$ if $s\in S(\eta_1,\eta_2;q)$, we have
\begin{equation} \label{inftypart 2}
\overline{\chi}_\infty(s)=\mbox{sgn}(\sigma_2(s))^{-u_2(\chi)} \left|\frac{\sigma_1(s)}{\sigma_2(s)} \right|^{-i \pi (m(\chi) +\rho(\chi))/\log \epsilon},  
\end{equation}
where $u_2(\chi)\in \{0,1\}$, $m(\chi)\in \mathbb{Z}$ and $\rho(\chi)$ and $g$ are as given in \eqref{rhodef} and \eqref{gdef}.
We are free to take $m(\chi)=0$ and $u_2(\chi)=0$, and hence $g=0$. With this choice, \eqref{inftypart 2} takes the form 
$$
\overline{\chi}_\infty(s)=\left| \frac{\sigma_1(s)}{\sigma_2(s)} \right|^{i\arg(\chi(\epsilon))/(2\log\epsilon)},
$$
and hence \eqref{defpsichi2} becomes 
\begin{equation}\label{transformed}
\psi_S(y,\chi)= \sum\limits_{s\in S(\eta_1,\eta_2;y)}\Lambda(s) \chi_{\text{Hecke}}((s))\left| \frac{\sigma_1(s)}{\sigma_2(s)} \right|^{i\arg(\chi(\epsilon))/(2\log\epsilon)}.
\end{equation}

\subsection{Picking out sectors} \label{sectorspicking}
Using the definitions of $S(\eta_1,\eta_2)$ and $S(\eta_1,\eta_2;y)$ in \eqref{def S} and \eqref{Seta1eta2x}, and recalling that $-1\le \eta_1<\eta_2\le 1$, the summation condition $s\in S(\eta_1,\eta_2;y)$ in \eqref{transformed} is equivalent to 
$$
\epsilon^{\eta_1}\le \left| \frac{\sigma_1(s)}{\sigma_2(s)} \right|\le \epsilon^{\eta_2} \quad \mbox{and} \quad s\in S(-1,1;y).
$$ 
Hence, we may write \eqref{transformed} in the form
\begin{equation}\label{transformed2}
\psi_S(y,\chi)= \sum\limits_{s\in S(-1,1;y)}\Lambda(s) \chi_{\text{Hecke}}((s))f_{\chi}(s),
\end{equation}
where 
$$
f_{\chi}(s):= \left| \frac{\sigma_1(s)}{\sigma_2(s)}\right|^{i\arg(\chi(\epsilon))/(2\log\epsilon)} I_{(\epsilon^{\eta_1},\epsilon^{\eta_2}]}\left(\left| \frac{\sigma_1(s)}{\sigma_2(s)} \right|\right).
$$
We will approximate $f(s)$ by a linear combination of terms of the form $\xi^n(s)$ with $n\in \mathbb{Z}$, where 
\begin{equation}\label{Def Xi}
\xi(\alpha):=\left| \frac{\sigma_1(\alpha)}{\sigma_2(\alpha)} \right|^{i\pi /\log\epsilon}\quad \mbox{for } \alpha\in \mathbb{K}.  
\end{equation}
We note that $\xi(\alpha)=\xi(\alpha u)$ for any unit $u$ in $\mathcal{O}_\mathbb{K}^{\ast}$, and thus 
\begin{equation*}
\xi_{\text{Hecke}}((\alpha)):=\xi(\alpha)
\end{equation*}
is well-defined and a Hecke character modulo $(1)$ on the principal fractional ideals $(\alpha)$. This will allow us to approximate $\psi_S(y,\chi)$ by a linear combination of sums with Hecke characters of the form
\begin{equation} \label{finallyHecke1}
\sum\limits_{s\in S(-1,1;y)}\Lambda(s) (\chi_{\text{Hecke}} \xi_{\text{Hecke}}^n) ((s)).
\end{equation}
Moreover, using Lemma \ref{sigma12}, the sum in \eqref{finallyHecke1} equals 
\begin{equation} \label{finallyHecke2}
\sum\limits_{\substack{\mathfrak{s}\in \mathcal{P}_{\mathbb{K}}\\ \mathcal{N}(\mathfrak{s})\le y}}\Lambda(\mathfrak{s}) (\chi_{\text{Hecke}} \xi_{\text{Hecke}}^n) (\mathfrak{s}).
\end{equation}

We set
\begin{equation} \label{Def Tau}
   \tau=\tau_{\chi}:=\frac{\arg(\chi(\epsilon))}{2\pi}\in \big(-1/2,1/2\big] 
\end{equation}
and
\begin{equation}\label{def W}
W=W(s):= \frac{\log|\sigma_1(s)/\sigma_2(s)|}{2\log\epsilon}.   
\end{equation}
We note that 
$$
\epsilon^{\eta_1}<\left|\frac{\sigma_1(s)}{\sigma_2(s)}\right|\le \epsilon^{\eta_2} \Longleftrightarrow \frac{\eta_1}{2}<W\le \frac{\eta_2}{2}
$$
and hence
\begin{equation} \label{fF}
f_{\chi}(s)=e(\tau W)I_{(\eta_1/2,\eta_2/2]}(W)=:F_{\chi}(W).
\end{equation}
We now write $F(W)=F_{\chi}(W)$ for brevity and view $F(W)$ as a function in an indeterminate $W$. We note that $F(W)=0$ if $W\in (-1/2,1/2]\setminus (\eta_1/2,\eta_2/2]$. On the interval $(-1/2,1/2]$, we may develop $F(W)$ into a Fourier series 
$$
\sum\limits_{n\in \mathbb{Z}}a_ne(nW),
$$
where 
\begin{equation*} 
a_n=a_n(\chi) =\int\limits_{-1/2}^{1/2}e(\tau z)I_{(\eta_1/2,\eta_2/2]}(z)e(-nz)\mbox{d}z= \int\limits_{\eta_1/2}^{\eta_2/2}e((\tau-n)z)\mbox{d}z.
\end{equation*} 
The equation 
\begin{equation}\label{FouriSeries}
F(W)=\sum\limits_{n\in \mathbb{Z}}a_ne(nW)
\end{equation}
is valid for all $W\in (1/2,1/2]$ except for $W=\eta_1/2,\eta_2/2$, the discontinuities of the function $F(W)$. For $n\neq \tau$, we calculate that 
\begin{equation*}
a_n = \frac{e((\tau - n) (\eta_2/2)) - e((\tau - n) (\eta_1/2))}{2\pi i (\tau - n)}.
\end{equation*}
If $n=\tau$, then since $\tau\in (-1/2,1/2]$, we have $n=0$ and  $a_n=a_0=1$.
We note that 
\begin{equation} \label{bound an}
 a_n\ll \frac{1}{1+|n|} \quad \mbox{for all } n\in \mathbb{Z},  
\end{equation}
where the implied constant is absolute.
The following lemma provides a truncated version of the above Fourier series expansion. 
\begin{lemma}\label{fourierapprox}
For any $Z\ge 2$ and $W\in (1/2,1/2]$, we have
\begin{equation} \label{fzcutoff}
F(W)= \sum\limits_{|n|\le Z} a_n e(nW) + O(T_Z(W)), 
\end{equation}
where
$$
T_Z(W):=\min\left\{\log Z, \frac{1}{Z||W-\eta_1/2||}+ \frac{1}{Z||W-\eta_2/2||}\right\}.
$$
Above, the implied constant is absolute, and $||X||$ denotes the distance of $X\in \mathbb{R}$ to the nearest integer. 
\end{lemma}
\begin{proof} This can be proved in a similar way as \cite[Lemma 2]{BDM24} by standard partial summation arguments. 
\end{proof}

We note that 
$$
e(nW(s))=\xi^n(s)=\xi_{\text{Hecke}}^n((s))
$$
and hence Lemma \ref{fourierapprox} and \eqref{fF} imply that 
$$
f_{\chi}(s)=\sum\limits_{|n|\le Z} a_n(\chi) \xi_{\text{Hecke}}^n((s)) + O(T_Z(W(s))).
$$
Combining this with \eqref{transformed2} and recalling the equality between the sums in \eqref{finallyHecke1} and \eqref{finallyHecke2}, we deduce that 
\begin{equation} \label{psiapprox}
\psi_S(y,\chi)= \sum\limits_{|n|\le Z} a_n(\chi) \sum\limits_{\substack{\mathfrak{s}\in \mathcal{P}_{\mathbb{K}}\\ \mathcal{N}(\mathfrak{s})\le y}}\Lambda(\mathfrak{s}) (\chi_{\text{Hecke}} \xi_{\text{Hecke}}^n) (\mathfrak{s})+O(\Sigma),
\end{equation}
where
$$
\Sigma:=\sum\limits_{s\in S(-1,1;x)} T_Z(W(s)).
$$

\subsection{Estimation of $\Sigma$}
To estimate $\Sigma$, we need to obtain information about the spacing of $W(s)$, as $s$ runs over the sector $S(-1,1;x)$.  Our method will be an extension of that in \cite[subsection 4.2.4.]{BDM24} where $s$ was restricted to prime elements. We divide the sector $S(-1,1;x)$, defined as in \eqref{Seta1eta2x}, into $O(\log x)$ dyadic parts 
\begin{equation} \label{thedefofPq}
\mathcal{S}(H):=\left\{s\in S(-1,1): H< \mathcal{N}(s)\le 2H\right\},
\end{equation}
where $H\ge 1$. 
We further split $\mathcal{S}(H)$ into three disjoint sets
\begin{equation}
\begin{split}
\mathcal{S}_0(H):= & \mathcal{S}(H)\cap \mathbb{Z},\\
\mathcal{S}_+(H):= & \{s\in \mathcal{S}(H)\setminus \mathbb{Z} : \sigma_1(s)/\sigma_2(s)>0\},\\
\mathcal{S}_-(H):= & \{s\in \mathcal{S}(H)\setminus \mathbb{Z} : \sigma_1(s)/\sigma_2(s)<0\}. 
\end{split}
\end{equation}
Now our goal is to estimate the sums
$$
\Sigma_0(H):=\sum\limits_{s\in \mathcal{S}_0(H)} T_Z(W(s)), \quad 
\Sigma_+(H):=\sum\limits_{s\in \mathcal{S}_+(H)} T_Z(W(s)), \quad
\Sigma_-(H):=\sum\limits_{s\in \mathcal{S}_-(H)} T_Z(W(s)). 
$$
If $s\in \mathcal{S}_0(H)$, then $s\in \mathbb{Z}$ and so $\mathcal{N}(s)=s^2$. Hence, using $T_Z(W(s))\le \log Z$, we have 
\begin{equation} \label{SN 0}
\Sigma_0(H) \ll H^{1/2}\log Z.
\end{equation}  
In the following, we estimate the sum $\Sigma_+(H)$. The sum $\Sigma_-(H)$ can be treated in an analogous way. 
Below we will consider only the case when $\mathbb{K}=\mathbb{Q}(\sqrt{d})$ with $d>1$ square-free and $d\equiv 2,3\bmod 4$ in which $\mathcal{O}_{\mathbb{K}}=\{u+v\sqrt{d}: u,v\in \mathbb{Z}\}$. The case when $d\equiv 1 \bmod{4}$ can be handled in a similar way.

For $s\in \mathcal{S}_+(H)$, let $A(s)$ be the number of $s'\in \mathcal{S}_+(H)$ such  that $W(s)=W(s')$ and set
\begin{equation} \label{mdef}
m:=\max\limits_{s\in \mathcal{S}_+(H)} A(s).
\end{equation}
We partition $\mathcal{S}_+(H)$ into $m$ subsetsets $\mathcal{M}$ such that $W(s_1)\not=W(s_2)$ for all $s_1, s_2 \in \mathcal{M}$ with $s_1\neq s_2$. Then, for each $\mathcal{M}$, we have
\begin{equation} \label{sumoverM} 
\sum\limits_{s\in \mathcal{M}} T_Z(W(s)) \ll \log Z+\frac{|\log l_+|}{Zl_+}, 
\end{equation}
where 
$$
l_+:=\min\limits_{\substack{s_1,s_2\in \mathcal{M}\\ s_1\not=s_2}} \left| W(s_1)-W(s_2)\right|.
$$
In the following, we establish a lower bound for $l_{+}$. 

Let $s_1,s_2\in \mathcal{M}$. Recalling the definition of $W(s)$ in \eqref{def W}, we see that
\begin{equation} \label{Wdif}
W(s_1)-W(s_2)= \frac{1}{2\log \epsilon}\cdot \log \frac{\sigma_1(s_1)\sigma_2(s_2)}{\sigma_2(s_1)\sigma_1(s_2)}. 
\end{equation}
Since $W(s_1)\not= W(s_2)$, it follows that 
\begin{equation} \label{nonzero}
\sigma_1(s_1)\sigma_2(s_2)-\sigma_2(s_1)\sigma_1(s_2)\not=0.
\end{equation}
We note that the bound 
$$
|\log x|\ge c|x-1|
$$
holds for a suitable constant $c=c(\epsilon)>0$ in the interval $[\epsilon^{-2},\epsilon^2]$, and hence,
$$
|W(s_1)-W(s_2)|\ge \frac{c}{2\log \epsilon}\cdot \frac{|\sigma_1(s_1)\sigma_2(s_2)-\sigma_2(s_1)\sigma_1(s_2)|}{|\sigma_2(s_1)\sigma_1(s_2)|}. 
$$ 
For $i=1,2$, let $s_i=a_i+b_i\sqrt{d}$ with $a_i,b_i\in \mathbb{Z}$. We calculate that 
$$
|\sigma_1(s_1)\sigma_2(s_2)-\sigma_2(s_1)\sigma_1(s_2)|=|2(b_1a_2-b_2a_1)\sqrt{d}|.
$$ 
This together with \eqref{nonzero} implies 
$$
|\sigma_1(s_1)\sigma_2(s_2)-\sigma_2(s_1)\sigma_1(s_2)|\ge 2\sqrt{d}.
$$
Using Lemma \ref{sigma12}, we deduce that
$$
|W(s_1)-W(s_2)|\ge \frac{c\sqrt{d}}{|\sigma_2(s_1)\sigma_1(s_2)|\log \epsilon}\ge 
 \frac{c\sqrt{d}}{\mathcal{N}(s_1s_2)^{1/2}\epsilon\log \epsilon}
\ge \frac{c\sqrt{d}}{2H\epsilon \log \epsilon}
$$
for any $s_1,s_2\in \mathcal{M}$ with $s_1\not=s_2$, and hence,
$$
l_+\ge \frac{c\sqrt{d}}{2H\epsilon \log \epsilon}.
$$
Recalling \eqref{sumoverM}, it follows that
\begin{equation*} 
\sum\limits_{s\in \mathcal{M}} T_Z(W(s)) \ll \log Z+\frac{H\log H}{Z} 
\end{equation*}
and hence,
\begin{equation} \label{SN b}
\Sigma_+(H) \ll m\left(\log Z+\frac{H\log H}{Z}\right),
\end{equation} 
with $m$ as defined in \eqref{mdef}. It remains to bound the said quantity $m$.
 
Fix $s_2 =u+v\sqrt{d}\in \mathcal{S}_+(H)$ with $u,v\in \mathbb{Z}$. We need to bound the number $A(s_2)$ of elements $s_1\in \mathcal{S}_+(H)$ such that $W(s_1)=W(s_2)$.  Using \eqref{Wdif}, we have
\begin{equation*}
W(s_1)=W(s_2)\Longleftrightarrow \sigma_1\left(\frac{s_1}{s_2}\right)=\sigma_2\left(\frac{s_1}{s_2}\right)
\Longleftrightarrow \frac{s_1}{s_2}\in \mathbb{Q}.   
\end{equation*}
So if $s_1,s_2\in \mathcal{S}_+(H)$ satisfy $W(s_1)=W(s_2)$, then $s_1/s_2=\mu_1/\mu_2$ for some $\mu_1,\mu_2\in \mathbb{Z}$ such that $\mu_2>0$ and $(\mu_1,\mu_2)=1$. We further observe that necessarily, $\mu_2|(u,v)$. Moreover, since 
$$
\frac{1}{4}\le \mathcal{N}\left(\frac{s_1}{s_2}\right)\le 4,
$$
we have
$$
\frac{\mu_2}{2}\le \mu_1\le 2\mu_2.
$$
Consequently,
$$
A(s_2) \ll \sum\limits_{\mu_2|(u,v)}\mu_2 \ll (u,v)^{1+\varepsilon}\ll H^{1/2+\varepsilon},
$$
which implies
\begin{equation} \label{mbound}
m\ll H^{1/2+\varepsilon}.
\end{equation}
Combining \eqref{SN b} and \eqref{mbound}, we get
\begin{equation*}  
\Sigma_+(H) \ll H^{1/2+\varepsilon}\left(\log Z+\frac{H\log H}{Z}\right).
\end{equation*}
In a similar way, we obtain a bound of the same quality for $\Sigma_-(H)$. Hence, taking \eqref{SN 0} into account, we obtain the estimate 
 \begin{equation*} 
\Sigma_0(H)+\Sigma_+(H)+\Sigma_-(H) \ll H^{1/2+\varepsilon}\left(\log Z+\frac{H\log H}{Z}\right),
\end{equation*}  
from which we deduce that
$$
\Sigma\ll x^{1/2+\varepsilon}\left(\log Z+\frac{x\log x}{Z}\right).
$$
Throughout the sequel, we take  $Z:=x$, which will be fine for our purposes. Thus we arrive at the bound
\begin{equation} \label{finalSigmabound}
\Sigma\ll x^{1/2+2\varepsilon}.
\end{equation}

\subsection{Application of the Siegel-Walfisz theorem} \label{SWT}
Combining \eqref{E*primitive}, \eqref{bound an}, \eqref{psiapprox} and \eqref{finalSigmabound}, and recalling our choice $Z=x$, we obtain 
\begin{equation}\label{intermediate}
\begin{split}
\sum\limits_{2\le \mathcal{N}(\mathfrak{q})\le Q} E^\ast(x;\mathfrak{q})\ll &  (\log x)\sum\limits_{2\le \mathcal{N}(\mathfrak{q})\le Q}\frac{1}{\varphi(\mathfrak{q})} \;\; \sideset{}{^\ast}\sum\limits_{\chi \bmod \mathfrak{q}}\ \sum\limits_{|n|\le x} \frac{1}{1+|n|} \times\\
& \max\limits_{y\le x} \Bigg|\sum\limits_{\substack{\mathfrak{s}\in \mathcal{P}_{\mathbb{K}}\\ \mathcal{N}(\mathfrak{s})\le y}}\Lambda(\mathfrak{s}) (\chi_{\text{Hecke}} \xi_{\text{Hecke}}^n) (\mathfrak{s})\Bigg|+Qx^{1/2+2\varepsilon}.
\end{split}
\end{equation}
Set $Q_0:=\log^B x$ where $B>0$ is arbitrarily given. If $2\le \mathcal{N}(\mathfrak{q})\le Q_0$, then we write 
\begin{equation} \label{cgchar}
\sum\limits_{\substack{\mathfrak{s}\in \mathcal{P}_{\mathbb{K}}\\ \mathcal{N}(\mathfrak{s})\le y}}\Lambda(\mathfrak{s}) (\chi_{\text{Hecke}} \xi_{\text{Hecke}}^n) (\mathfrak{s})= 
\frac{1}{h}\cdot \sum\limits_{\psi\in \mathcal{X}} \sum\limits_{\substack{\mathfrak{s}\in \mathcal{I}_{\mathbb{K}}\\ \mathcal{N}(\mathfrak{s})\le y}}\Lambda(\mathfrak{s}) (\chi_{\text{Hecke}} \xi_{\text{Hecke}}^n\psi) (\mathfrak{s}),
\end{equation}
where $\mathcal{X}$ is the group of class group characters, and use Proposition \ref{Siegel} to bound the character sum on the right-hand side. In this way, the contribution of the moduli $\mathfrak{q}$ with $2\le \mathcal{N}(\mathfrak{q})\le Q_0$ to the right-hand side of \eqref{intermediate} is easily bounded by 
\begin{equation}\label{smallmodcont}
(\log x)\sum\limits_{2\le \mathcal{N}(\mathfrak{q})\le Q_0}\frac{1}{\varphi(\mathfrak{q})} \;\; \sideset{}{^\ast}\sum\limits_{\chi \bmod \mathfrak{q}}\sum\limits_{|n|\le Z} \frac{1}{1+|n|} \cdot
 \max\limits_{y\le x} \Bigg|\sum\limits_{\substack{\mathfrak{s}\in \mathcal{P}_{\mathbb{K}}\\ \mathcal{N}(\mathfrak{s})\le y}}\Lambda(\mathfrak{s}) (\chi_{\text{Hecke}} \xi_{\text{Hecke}}^n) (\mathfrak{s})\Bigg|
\ll_{A,B} \frac{x}{\log^A x}.
\end{equation}
Recall the assumption $Q\le x^{4/5}$ from Theorem \ref{main}. We split the remaining contribution into $O(\log^2 x)$ dyadic subsums $D(P,N)$ accounting for moduli $\mathfrak{q}$ satisfying $P<\mathcal{N}(\mathfrak{q}) \le 2P$ with $Q_0\le P\le Q$ and integers $n$ satisfying $N\le |n| \le 2N$ with $0\le N\le x$. If $Q\le x$, then using \eqref{cgchar} and taking the well-known bound 
$$
\frac{1}{\varphi(\mathfrak{q})}\ll \frac{\log\log \mathcal{N}(\mathfrak{q})}{\mathcal{N}(\mathfrak{q})} 
$$
into account, these subsums are bounded by 
\begin{equation}\label{new sum}
D(P,N)\ll \frac{\log^2 x}{P(1+N)} \sum\limits_{\tilde{\chi}\in \mathcal{F}(P,N)}\max\limits_{y\le x}\Big|\sum\limits_{\substack{\mathfrak{s}\in \mathcal{I}_{\mathbb{K}}\\ \mathcal{N}(\mathfrak{s}) \le y}} \Lambda(\mathfrak{s})\tilde{\chi}(\mathfrak{s})\Big|,
\end{equation}
where 
\begin{equation}\label{Family}
\begin{split}
\mathcal{F}(P,N):= \{\chi_{\text{Hecke}}\xi^n_{\text{Hecke}}\psi: & \ \psi\in \mathcal{X}, \mbox{ and } \chi \mbox{ is a primitive Dirichlet character modulo } \mathfrak{q} \\ & \mbox{ with } P<\mathcal{N}(\mathfrak{q})\le 2P \mbox{ and } N\le |n|\le 2N\}.
\end{split}
\end{equation}
Here it is understood that every Hecke character $\chi_{\text{Hecke}}$ belongs to precisely one primitive Dirichlet character $\chi$. Following usual terminology, we refer to $T_I$ as a type I and $T_{II}$ as a type II bilinear sum. We note that 
\begin{equation} \label{Fnote}
\sharp\mathcal{F}(P,N)\ll P^2(N+1).
\end{equation}
Throughout the sequel, we abbreviate the term $N+1$ as  
$$
\tilde{N}:=N+1.
$$

\section{Reduction to bilinear sums}
\subsection{Application of Vaughan's identity}
To transform the inner-most sum over $\mathfrak{s}$ on the right-hand side of \eqref{new sum}, we apply \eqref{Lambdasplit} with $\mathbf{G}(\mathfrak{s})=\tilde{\chi}(\mathfrak{s})I_{\mathcal{N}(\mathfrak{s})>U}$ and the bounds \eqref{theS1}, \eqref{theS2} and \eqref{theS3} to obtain
\begin{equation}\label{applyVaughan's}
\sum\limits_{\mathcal{N}(\mathfrak{s})\le y} \Lambda(\mathfrak{s})\tilde{\chi}(\mathfrak{s})=\sum\limits_{U<\mathcal{N}(\mathfrak{s})\le y} \Lambda(\mathfrak{s})\tilde{\chi}(\mathfrak{s})+O(U)= S_1(\tilde{\chi})+S_2(\tilde{\chi})+S_3(\tilde{\chi})+O(U),
\end{equation}
where $U,V\ge 1$ are free parameters, to be fixed later, and 
\begin{equation*}
 S_1(\tilde{\chi}) \ll (\log UV) \sum\limits_{ \mathcal{N}(\mathfrak{t})\le UV} \Bigg| \sum\limits_{\substack{ U/\mathcal{N}(\mathfrak{t})<\mathcal{N}(\mathfrak{r}) \le y/\mathcal{N}(\mathfrak{t})}} \tilde{\chi}(\mathfrak{t}\mathfrak{r}) \Bigg| 
\end{equation*}
\begin{equation*}
S_2(\tilde{\chi})\ll (\log y) \sum\limits_{\mathcal{N}(\mathfrak{l}) \le V} \max\limits_{w\le y/\mathcal{N}(l)} \Bigg| \sum\limits_{\substack{U/\mathcal{N}(\mathfrak{l})< \mathcal{N}(\mathfrak{r})\le w}} \tilde{\chi}(\mathfrak{l}\mathfrak{r})\Bigg|
\end{equation*}
and
\begin{equation*}
S_3(\tilde{\chi})=-\sum\limits_{U<\mathcal{N}(\mathfrak{m})\le y/ V} \quad \sum\limits_{\max\{U/\mathcal{N}(\mathfrak{m}),V\} < \mathcal{N}(\mathfrak{r}) \le y/ \mathcal{N}(\mathfrak{m})} \Lambda(\mathfrak{m})H(\mathfrak{r})\tilde{\chi}(\mathfrak{m}\mathfrak{r}),
\end{equation*}
where $H(\mathfrak{r})$ is defined as in \eqref{Hdef}. Recalling \eqref{Fnote}, this implies
\begin{equation}\label{T1+T2}
 \sum\limits_{\tilde{\chi}\in \mathcal{F}(P,N)}\max\limits_{y\le x}\Big|\sum\limits_{\mathcal{N}(\mathfrak{s})\le y} \Lambda(\mathfrak{s})\tilde{\chi}(\mathfrak{s})\Big|\ll T_I +T_{II}+P^2\tilde{N}U,
\end{equation}
where 
\begin{equation} \label{T1def}
T_I:=(\log x)\sum\limits_{\tilde{\chi}\in \mathcal{F}(P,N)}  \sum\limits_{ \mathcal{N}(\mathfrak{t})\le UV}\max\limits_{w\le x/\mathcal{N}(\mathfrak{t})} \Bigg| \sum\limits_{\substack{U/\mathcal{N}(\mathfrak{t})< \mathcal{N}(\mathfrak{r})\le w}} \tilde{\chi}(\mathfrak{r}) \Bigg|
\end{equation}
and 
\begin{equation} \label{T2def}
T_{II}:= \sum\limits_{\tilde{\chi}\in \mathcal{F}(P,N)} \max\limits_{y\le x}\Bigg|\sum\limits_{U<\mathcal{N}(\mathfrak{m})\le y/ V} \sum\limits_{\max\{U/\mathcal{N}(\mathfrak{m}),V\} < \mathcal{N}(\mathfrak{r}) \le y/ \mathcal{N}(\mathfrak{m})} \Lambda(\mathfrak{m})H(\mathfrak{r})\tilde{\chi}(\mathfrak{m}\mathfrak{r})\Bigg|,
\end{equation}
provided that $UV\le x$. Following usual terminology, we refer to $T_I$ as a type I and to $T_{II}$ as a type II sum. 

\subsection{Estimation of the type I sum}
By Proposition~\ref{Landau} with $n=2$, the inner-most sum over $\mathfrak{r}$ on the right-hand side of \eqref{T1def} is bounded by
$$
\sum\limits_{\substack{U/\mathcal{N}(\mathfrak{t})< \mathcal{N}(\mathfrak{r})\le w}} \tilde{\chi}(\mathfrak{r}) \ll w^{1/3} \mathcal{N}(\mathfrak{q})^{1/3}\log^2 \mathcal{N}(\mathfrak{q}).
$$
Plugging this bound into \eqref{T1def} and using \eqref{Fnote}, we deduce that
\begin{equation} \label{T1last}
\begin{split}
T_I &\ll (\log x)P^2\tilde{N} \sum\limits_{\mathcal{N}(\mathfrak{l}) \le UV}  x^{1/3}\mathcal{N}(\mathfrak{l})^{-1/3}P^{1/3}\log^2 2P \ll x^{1/3}P^{7/3}\tilde{N}(UV)^{2/3}\log^3 x,
\end{split}
\end{equation}
provided that $P\le x$.

\subsection{Tailoring the type II sum}
To estimate $T_{II}$, we decompose the $\mathfrak{m}$-sum into dyadic subsums. For any $M\in[1,x/V]$, we define 
$$
b_{\mathfrak{m}}:= \begin{cases} \Lambda(\mathfrak{m}) & \mbox{ if } \max (U,M)<\mathcal{N}(\mathfrak{m})\le 2M,\\
0 & \mbox{ otherwise} \end{cases}
$$
and
$$
c_{\mathfrak{r}}:= \begin{cases} H(\mathfrak{r}) & \mbox{ if } \mathcal{N}(\mathfrak{r})>V,\\
0 & \mbox{ otherwise.} \end{cases}
$$
Then 
\begin{equation} \label{TIIdom}
T_{II}\ll \sum\limits_{\substack{j\in \mathbb{N}_0\\ U\le 2^j\le x/V}} T_{II}(2^j),
\end{equation}
where 
\begin{equation}\label{T2 contribution}
T_{II}(M) := \sum\limits_{\tilde{\chi}\in \mathcal{F}(P,N)} \max\limits_{y\le x}\Big|\sum\limits_{\substack{\substack{\mathcal{N}(\mathfrak{m})\le 2M\\ \mathcal{N}(\mathfrak{m})\le y/V}}} \sum\limits_{\substack{\mathcal{N}(\mathfrak{r}) \le x/M \\ U<\mathcal{N}(\mathfrak{mr})\le y}} b_{\mathfrak{m}} c_{\mathfrak{r}} \tilde{\chi}(\mathfrak{m}\mathfrak{r})\Big|.
\end{equation}

\section{Large sieve inequality for Hecke characters in $\mathcal{F}(P,N)$}\label{sec large}
To estimate the type II sum in \eqref{T2 contribution}, we first establish the following large sieve inequality for the family of characters $\mathcal{F}(P,N)$. 
\begin{lemma}\label{LSieve}
Let $\mathcal{F}(P,N)$ be defined as in \eqref{Family}, $K\ge 1$ and $a_{\mathfrak{k}}\in \mathbb{C}$ for any ideal $\mathfrak{k}\in \mathcal{P}_{\mathbb{K}}$. Then we have
\begin{equation} \label{Lsieve2}
\sum\limits_{\tilde{\chi}\in F(P,N)} \Big|\sum\limits_{\mathcal{N}(\mathfrak{k})\le K} a_{\mathfrak{k}} \tilde{\chi}(\mathfrak{k})\Big|^2\ll \left(K + P^3\tilde{N}(PK)^\varepsilon\right) \sum\limits_{\mathcal{N}(\mathfrak{k})\le  K} |a_{\mathfrak{k}}|^2.
\end{equation}
\end{lemma}
\begin{proof}
Applying the duality principle (Lemma~\ref{Duality}), it suffices to show that for any complex sequence $(b_{\tilde{\chi}})$, we have 
\begin{equation} \label{Lsieve3}
S:= \sum\limits_{\mathcal{N}(\mathfrak{k})\le K} \Big|\sum\limits_{\tilde{\chi}\in F(P,N)} b_{\tilde{\chi}} \tilde{\chi}(\mathfrak{k})\Big|^2\ll \big(K+ P^3\tilde{N}(PK)^\varepsilon\big) \, \sum\limits_{\tilde{\chi}\in F(P,N)} |b_{\tilde{\chi}}|^2. 
\end{equation} 
Set 
$$
\omega(x):=\exp\left(-\frac{x^2}{K^2}\right).
$$ 
Then 
\begin{equation*}
S\ll \sum\limits_{\mathfrak{k}\in \mathcal{I}_{\mathbb{K}}} \omega\left(\mathcal{N}(\mathfrak{k})\right)\cdot
\Big|\sum\limits_{\tilde{\chi}\in F(P,N)} b_{\tilde{\chi}} \tilde{\chi}(\mathfrak{k})\Big|^2.
\end{equation*} 
Opening the square and interchanging the order of summations, we deduce that 
\begin{equation} \label{Ses}
\begin{split}
S \ll & \sum\limits_{\tilde{\chi}_1,\tilde{\chi}_2\in F(P,N)} b_{\tilde{\chi}_1}\overline{b_{\tilde{\chi}}}_2 \sum\limits_{\mathfrak{k} \in \mathcal{I}_{\mathbb{K}}} \omega\left(\mathcal{N}(\mathfrak{k})\right)\tilde{\chi}_1\overline{\tilde{\chi}}_2(\mathfrak{k})\\
\ll & \sum\limits_{\tilde{\chi}_1,\tilde{\chi}_2\in F(P,N)} \big( |b_{\tilde{\chi}_1}|^2+ |b_{\tilde{\chi}_2}|^2\big)  \Big| \sum\limits_{\mathfrak{k} \in \mathcal{I}_{\mathbb{K}}} \omega\left(\mathcal{N}(\mathfrak{k})\right)\tilde{\chi}_1\overline{\tilde{\chi}}_2(\mathfrak{k}) \Big|\\
\ll & \sum\limits_{\tilde{\chi}_1\in F(P,N)} |b_{\tilde{\chi}_1}|^2 \sum\limits_{\tilde{\chi}_2\in F(P,N)} \Big|\sum\limits_{\mathfrak{k} \in \mathcal{I}_{\mathbb{K}}} \omega(\mathcal{N}(\mathfrak{k}))\tilde{\chi}_1\overline{\tilde{\chi}}_2(\mathfrak{k})\Big|\\
\ll & \Big(\sum\limits_{\tilde{\chi}_1\in F(P,N)} |b_{\tilde{\chi}_1}|^2 \Big) \max_{\tilde{\chi_1} \in F(P,N)} \sum\limits_{\tilde{\chi}_2\in F(P,N)} \Big| \sum\limits_{\mathfrak{k} \in \mathcal{I}_{\mathbb{K}}} \omega(\mathcal{N}(\mathfrak{k}))\tilde{\chi}_1\overline{\tilde{\chi}}_2(\mathfrak{k})\Big|.
\end{split}
\end{equation}
Clearly,
\begin{equation} \label{dia}
\Big| \sum\limits_{\mathfrak{k} \in \mathcal{I}_{\mathbb{K}}} \omega(\mathcal{N}(\mathfrak{k}))\tilde{\chi}_1\overline{\tilde{\chi}}_2(\mathfrak{k})\Big| \le \sum\limits_{\mathfrak{k} \in \mathcal{I}_{\mathbb{K}}} \omega(\mathcal{N}(\mathfrak{k})) \ll K \quad \mbox{if } \tilde{\chi}_1=\tilde{\chi}_2.
\end{equation}
If $\tilde{\chi}_1\not=\tilde{\chi}_2$, then applying the inverse Mellin transform, we may transform the above sum over $\mathfrak{k}$ into
\begin{equation} \label{Mellinapp} 
\begin{split}
\sum\limits_{\mathfrak{k} \in \mathcal{I}_{\mathbb{K}}} \omega(\mathcal{N}(\mathfrak{k}))
\tilde{\chi}_1\overline{\tilde{\chi}}_2(\mathfrak{k}) &= \frac{1}{2\pi i} \int\limits_{\Re (s)=\sigma}  \Tilde{\omega}(s) \sum\limits_{\mathfrak{k} \in \mathcal{I}_{\mathbb{K}}} \frac{\tilde{\chi}_1\overline{\tilde{\chi}}_2(\mathfrak{k})}{\mathcal{N}(\mathfrak{k})^s} ds \\
&= \frac{1}{2\pi i} \int\limits_{\Re (s)=\sigma} \Tilde{\omega}(s) L(s, \tilde{\chi}_1\overline{\tilde{\chi}}_2)ds,
\end{split}
\end{equation}  
where $\sigma$ is any real number greater than 1 and 
$$
\Tilde{\omega}(s):=\int\limits_{0}^{\infty}\omega(y)y^{s-1}dy=\frac{1}{2}\cdot K^s\Gamma\left(\frac{s}{2}\right)
$$
is the Mellin transform of $\omega$. Since both functions $\tilde{\omega}(s)$ and $L(s,\tilde{\chi}_1\overline{\tilde{\chi}}_2)$ are analytic in the half plane $\Re(s)>0$, the line of integration can be shifted to $\Re(s)=\varepsilon>0$. By the rapid decay of the Gamma function on vertical lines, we have   
\begin{equation} \label{tildeomegabound} 
\widetilde{\omega}(\varepsilon+it) \ll K^{\varepsilon}(1+|t|)^{-j}
\end{equation}
for any $j>0$. The convexity bound for Hecke $L$-functions (see \cite[Exercise 3 on page 100]{IwKo}) implies that
\begin{equation} \label{subcon}
L(\varepsilon+it, \tilde{\chi}_1\overline{\tilde{\chi}}_2)\ll {\bf f}(\tilde{\chi}_1\overline{\tilde{\chi}}_2)^{1/2+\varepsilon} (1+|t|)^{1+\varepsilon},
\end{equation}
where ${\bf f}(\tilde{\chi}_1\overline{\tilde{\chi}}_2)$ is the norm of the conductor of  $\tilde{\chi}_1\overline{\tilde{\chi}}_2$ (for information on the conductor, see \cite[section 3.3]{Miya}). Recalling \eqref{Family}, we have  
\begin{equation} \label{conductorbound}
{\bf f}(\tilde{\chi}_1\overline{\tilde{\chi}}_2)\ll P^2
\end{equation} 
if $\tilde{\chi}_1,\tilde{\chi}_2\in F(P,N)$. 
Using \eqref{Mellinapp}, \eqref{tildeomegabound}, \eqref{subcon} and \eqref{conductorbound} with $j>2+\varepsilon$, we have the bound
\begin{equation} \label{afterMellin}
\sum\limits_{\mathfrak{k} \in \mathcal{I}_{\mathbb{K}}} \omega(\mathfrak{k})\tilde{\chi}_1\overline{\tilde{\chi}}_2(\mathfrak{k})\ll P(PK)^{\varepsilon} \quad \mbox{if } \tilde{\chi}_1\not=\tilde{\chi}_2.
\end{equation}  
Combining \eqref{Fnote}, \eqref{Ses}, \eqref{dia} and \eqref{afterMellin}, we obtain the desired estimate \eqref{Lsieve3}. This completes the proof. 
\end{proof}

{\bf Remark 3:} Since $\sharp\mathcal{F}(P,N)\ll P^2\tilde{N}$, we conjecture that \eqref{Lsieve2} holds with the term $P^3\tilde{N}(PK)^{\varepsilon}$ replaced by $P^2\tilde{N}(P\tilde{N}K)^{\varepsilon}$. 

\section{Estimation of  type II sums}
Using Lemma \ref{LSieve} above, we now bound the type II sum $T_{II}(M)$.
\begin{lemma}\label{Bilinearsum} 
We have 
\begin{equation}\label{BilinearSum2}
\begin{split}
T_{II}(M) \ll & \Big(M+ P^{3} \tilde{N} (Px)^{\varepsilon}\Big)^{1/2}\Big(xM^{-1} + P^{3} \tilde{N} (Px)^{\varepsilon} \Big)^{1/2} \times \\
&\Big(\sum\limits_{\substack{\mathcal{N}(\mathfrak{m})\le 2M}} |b_{\mathfrak{m}}|^2\Big)^{1/2} \Big(\sum\limits_{\substack{\mathcal{N}(\mathfrak{r})\le x/M}}|c_{\mathfrak{r}}|^2\Big)^{1/2} \log^2 x.
\end{split}
\end{equation}
\end{lemma}

\begin{proof}
For any $K,L\ge 1$ and sequences $b_{\mathfrak{m}}$ and $c_{\mathfrak{r}}$ of complex numbers with indices $\mathfrak{m}$ and $\mathfrak{r}$ running over $\mathcal{I}_{\mathbb{K}}$, we have, using the Cauchy-Schwarz inequality and Lemma \ref{LSieve}, 
\begin{equation}\label{BilinearSum3}
\begin{split}
&\sum\limits_{\tilde{\chi}\in F(P,N)} \Big|\sum\limits_{\substack{\mathcal{N}(\mathfrak{m})\le K}} \sum\limits_{\substack{\mathcal{N}(\mathfrak{r}) \le L }} b_{\mathfrak{m}} c_{\mathfrak{r}} \tilde{\chi}(\mathfrak{m}\mathfrak{r})\Big| \\
\le & \Big(\sum\limits_{\tilde{\chi}\in F(P,N)} \Big|\sum\limits_{\substack{\mathcal{N}(\mathfrak{m})\le K}} b_\mathfrak{m}\tilde{\chi}(\mathfrak{m})\Big|^2\Big)^{1/2} \Big(\sum\limits_{\tilde{\chi}\in F(P,N)} \Big|\sum\limits_{\substack{1\le \mathcal{N}(\mathfrak{r})\le L}} c_\mathfrak{r}\tilde{\chi}(\mathfrak{r})\Big|^2\Big)^{1/2}\\\
\ll & \left(K + P^{3}\tilde{N} (PK)^{\varepsilon} \right)^{1/2}\left(L + P^{3}\tilde{N} (PL)^{\varepsilon}\right)^{1/2}\times\\ &  \Big(\sum\limits_{\substack{\mathcal{N}(\mathfrak{m})\le K}} |b_{\mathfrak{m}}|^2\Big)^{1/2} \Big(\sum\limits_{\substack{\mathcal{N} (\mathfrak{r})\le L}}|c_{\mathfrak{r}}|^2\Big)^{1/2}.
\end{split}
\end{equation}
The term $T_{II}(M)$ defined in \eqref{T2 contribution} takes the same form as the sum in the first line above with $K=2M$ and $L=x/M$, except that there are two extra summation conditions $\mathcal{N}(\mathfrak{m})\le y/V$ and $U<\mathcal{N}(\mathfrak{mr})\le y$ and an extra maximum over $y\le x$. These extra  conditions and maximum can be removed in a standard way using Lemma \ref{LemmaCosmetic} (cosmetic surgery), reducing the estimation of $T_{II}(M)$ to \eqref{BilinearSum3} (for details, see \cite[proof of Satz 5.5.2.]{Bru}, for example). This comes at the cost of an extra factor of $\log^2 x$, as compared to the bound in \eqref{BilinearSum3}. Hence, we arrive at the desired estimate \eqref{BilinearSum2}.
\end{proof}
Noting the well-known bounds 
$$
\sum\limits_{\substack{\mathcal{N}(\mathfrak{m})\le K}} |b_{\mathfrak{m}}|^2\le \sum\limits_{\substack{\mathcal{N}(\mathfrak{m})\le K}} \Lambda({\mathfrak{m}})^2\ll K\log 2K
$$
and 
$$
\sum\limits_{\mathcal{N}(\mathfrak{r})\le L}|c_{\mathfrak{r}}|^2\le \sum\limits_{\mathcal{N}(\mathfrak{r})\le L}\tau(\mathfrak{r})^2\ll L \cdot \log^3 2L,
$$
where $\tau(\mathfrak{r})$ denotes the number of ideal divisors of $\mathfrak{r}$, we deduce that
\begin{equation} \label{TMfinal}
\begin{split}
T_{II}(M) \ll & \left(M^{1/2} + P^{3/2} \tilde{N}^{1/2} (Px)^{\varepsilon/2} \right) \left(x^{1/2}M^{-1/2} + P^{3/2} \tilde{N}^{1/2} (Px)^{\varepsilon/2} \right)\times\\ & (M\log x)^{1/2} (xM^{-1}\log^3 x)^{1/2} \log^2 x \\
\ll & \Big(x+xP^{3/2}\tilde{N}^{1/2}M^{-1/2}(Px)^{\varepsilon/2}+x^{1/2}P^{3/2}\tilde{N}^{1/2}M^{1/2}(Px)^{\varepsilon/2}\\
& + x^{1/2}P^{3}\tilde{N}(Px)^{\varepsilon}\Big) \log^{4} x.
\end{split}
\end{equation}
From \eqref{TIIdom} and \eqref{TMfinal}, we infer the bound
\begin{equation}\label{T2last}
T_{II}\ll\Big(x+xP^{3/2}\tilde{N}^{1/2}\left(U^{-1/2}+V^{-1/2}\right)(Px)^{\varepsilon/2}+ x^{1/2}P^{3}\tilde{N}(Px)^{\varepsilon}\Big) \log^{7} x,
\end{equation}
provided that $P\le x$. 

\section{Proof of Theorem \ref{main}}
Combining \eqref{new sum}, \eqref{T1+T2}, \eqref{T1last} and \eqref{T2last}, we obtain
 \begin{equation}\label{Finalbound}
\begin{split}
D(P,N) \ll  & \Big(xP^{-1}\tilde{N}^{-1}+xP^{1/2}\tilde{N}^{-1/2}(U^{-1/2}+V^{-1/2})(Px)^{\varepsilon/2}+ \\ & x^{1/2}P^{2}(Px)^{\varepsilon}+x^{1/3}P^{4/3}(UV)^{2/3}+PU\Big) \log^9 x   
\end{split}
\end{equation}
if $P\le x$. Now we assume, more restrictively, that $P\le x^{4/5}$ and pick
$$
U:=V:=x^{4/11}P^{-5/11},
$$
noting that $U,V\ge 1$ under this assumption.  Under this choice, \eqref{Finalbound} implies
\begin{equation*}
D(P,N) \ll\left(xP^{-1}+ \left(x^{9/11}P^{8/11}+ x^{1/2}P^{2}\right)(Px)^{\varepsilon}\right) \log^9 x.   
\end{equation*}
Recalling the considerations in subsection \ref{SWT} and fixing $B:=A+11$, we conclude that 
$$
\sum\limits_{2\le \mathcal{N}(\mathfrak{q})\le Q} E^\ast(x;\mathfrak{q})\ll \left(x^{9/11}Q^{8/11}+ x^{1/2}Q^{2}\right) x^{2\varepsilon}+\frac{x}{\log^{A} x}
$$
if $Q\le x^{4/5}$. Recalling the notations in subsection \ref{Picking}, this is the same as the bound
\begin{equation} \label{same}
\sum\limits_{\mathcal{N}(\mathfrak{q})\le Q}\,\, \max\limits_{\substack{a\bmod \mathfrak{q} \\(a,\mathfrak{q})=1}} \,\, \max\limits_{y\le x}\left| \tilde{E}_S(y;\mathfrak{q},a) \right| \ll \left(x^{9/11}Q^{8/11}+ x^{1/2}Q^{2}\right) x^{\varepsilon}+\frac{x}{\log^A x}
\end{equation}
upon redefining $\varepsilon$. 
Thus, for the proof of Theorem \ref{main}, it suffices to approximate $\tilde{E}_S(y;\mathfrak{q},a)$ sufficiently closely by $E_S(y;\mathfrak{q},a)$. To this end, we need to appoximate the main term
$$
M_S(y;\mathfrak{q},a)=\frac{\eta}{\varphi(\mathfrak{q})} \sum\limits_{\substack{s\in S(-1,1;y)\\ (s,\mathfrak{q})=1}} \Lambda(s)
$$
by 
$$
\tilde{M}_S(y;\mathfrak{q},a):=\frac{1}{\varphi(\mathfrak{q})} \sum\limits_{\substack{s\in S(\eta_1,\eta_2;y)\\ (s,\mathfrak{q})=1}} \Lambda(s),
$$
which can be done in a similar way as in subsection \ref{sectorspicking} by picking out the sector condition $s\in S(\eta_1,\eta_2;x)$ using Hecke characters of the form $\xi_{\text{Hecke}}^n$. This results in an appoximation of the form
$$
M_S(y;\mathfrak{q},a)=\tilde{M}_S(y;\mathfrak{q},a)+O\left(\frac{x}{\varphi(\mathfrak{q})\log^{A+1} x}\right)
$$
and hence,
$$
E_S(y;\mathfrak{q},a)=\tilde{E}_S(y;\mathfrak{q},a)+O\left(\frac{x}{\varphi(\mathfrak{q})\log^{A+1} x}\right),
$$
which clearly suffices to deduce \eqref{th main eq1} from \eqref{same}. This completes the proof of Theorem \ref{main}. \\ \\
{\bf Remark 4.} The sole reason that we are not able to handle small sectors with $\eta= x^{-\delta}$ for $\delta$ in a suitable interval $0<\delta\le \delta_0$ along the above lines is the following.  Our treatment of small moduli with $\mathcal{N}(\mathfrak{q})\le Q_0=\log^A x$ making direct use of the Siegel-Walfisz theorem does not give a bound better than $O(x\exp(-C\sqrt{\log x}))$ for their total contribution, no matter what the size of $\eta$ is. However, we need a bound of order of magnitude $O(\eta x\log^{-A} x)$ if $\eta$ is allowed to depend on $x$. To resolve this issue, we may resort to using the explicit formula and zero-density estimates for Hecke $L$-functions, which requires more machinery. 
Of course, large sieve inequalities for Hecke characters (more generally, mean value estimates for Dirichlet polynomials twisted with Hecke characters) are needed here as well.
\\ \\
{\bf Acknowledgements.} Both authors would like to thank the anonymous referee for valuable comments, in particular, for pointing out a way to improve our initial level of distribution $1/5-\varepsilon$ to $1/4-\varepsilon$ using a refined treatment of the large sieve with Hecke characters. The first-named author would like to thank the Ramakrishna Mission Vivekananda Educational and Research Institute for an excellent work environment. The second-named author is grateful to the Indian Statistical Institute Kolkata for its excellent research environment and acknowledges support from the Research Associate Fellowship at the Institute.

\end{document}